\documentclass[11pt]{amsart}
\usepackage[utf8]{inputenc}

\usepackage{amsmath}
\usepackage{amsthm}
\usepackage{accents}
\usepackage{amssymb}
\usepackage{graphicx}
\usepackage{color}
\usepackage{enumerate}
\usepackage[normalem]{ulem}

\newcommand{\del}{\partial}
\newcommand{\CL}{\mathcal{L}}

\newcommand{\R}{\mathbb{R}}
\renewcommand{\S}{\mathbb{S}}

\newcommand{\Scal}{\operatorname{R}}
\newcommand{\C}{\operatorname{\mathcal{C}}}

\newcommand{\Ric}{\operatorname{Rc}}
\newcommand{\supp}{\operatorname{supp}}

\newcommand{\eps}{\varepsilon}

\newcommand {\mc} {H}

\renewcommand {\H} {\mathcal{H}}
\newcommand {\dist} {\operatorname{dist}}

\newtheorem{theorem}{Theorem}[section]
\newtheorem{lemma}[theorem]{Lemma}
\newtheorem{definition}[theorem]{Definition}
\newtheorem{corollary}[theorem]{Corollary}
\newtheorem{proposition}[theorem]{Proposition}

\title[Large isoperimetric surfaces center]{Large isoperimetric surfaces in
initial data sets}

\author{Michael Eichmair \and Jan Metzger}

\address{Michael Eichmair, ETH Z\"urich, Departement Mathematik, 8092 Z\"urich, Switzerland}
\email{michael.eichmair@math.ethz.ch}

\address{Jan Metzger, Universit\"at Potsdam, Institut f\"ur Mathematik, Am Neuen Palais 10, 14469 Potsdam, Germany}
\email{jan.metzger@uni-potsdam.de}

\thanks{Michael Eichmair gratefully acknowledges the support of the NSF grant DMS-0906038 and of the SNF grant 2-77348-12.}

\date{\today}

\begin{document}

\begin{abstract} We study the isoperimetric structure of asymptotically flat Riemannian $3$--manifolds
$(M, g)$ that are $\C^0$-asymptotic to Schwarzschild of mass $m>0$. 
Refining an argument due to H. Bray we obtain an effective volume comparison theorem 
in Schwarzschild. We use it to show that isoperimetric regions exist in
$(M, g)$ for all sufficiently large volumes, and that they are close to
 centered coordinate spheres. This implies that the volume-preserving stable
constant mean curvature spheres constructed by G. Huisken and S.-T. Yau as well as R. Ye as perturbations of large centered coordinate spheres 
minimize area among all competing surfaces that enclose the same volume. This confirms a
conjecture of H. Bray. Our results are consistent with the uniqueness
results for volume-preserving stable constant mean curvature surfaces in initial data sets obtained by G. Huisken and S.-T. Yau and strengthened by J. Qing and
G. Tian. The additional hypotheses  that the surfaces be spherical and far out in the asymptotic region in their results are not necessary in our work. 
\end{abstract} 

\maketitle

%%%%%%%%%%%%%%%%%%%%%%%%%%%%%%%%%%%%%%%%%%%%%%%%%%%%%%%%%%%%%%%%%%%%%%%%%%%%%%%%%%%%%%%
%%%%%%%%%%%%%%%%%%%%%%%%%%%%%%%%%%%%%%%%%%%%%%%%%%%%%%%%%%%%%%%%%%%%%%%%%%%%%%%%%%%%%%%
%%%%%%%%%%%%%%%%%%%%%%%%%%%%%%%%%%%%%%%%%%%%%%%%%%%%%%%%%%%%%%%%%%%%%%%%%%%%%%%%%%%%%%%

\section{Introduction}

In this paper we describe completely the large isoperimetric surfaces
of asymptotically flat Riemannian $3$-manifolds $(M, g)$ that are
$\C^0$-asympto\-tic to the Schwarzschild metric of mass $m>0$. Such
Riemannian manifolds arise naturally as initial data for the
time-symmetric Cauchy problem for the Einstein equations in general
relativity. For brevity we will refer to such $(M, g)$ as initial data
sets in the introduction.

A special case of the singularity theorem of Hawking and Penrose
asserts that the future spacetime development of time-symmetric
initial data for the Einstein equations that contains a closed minimal
surface is causally incomplete. As a trial for cosmic censorship,
R. Penrose suggested that the area of an outermost minimal surface in
an initial data set should provide a lower bound for the ADM-mass of
the spacetime development of the initial data set. The ``Penrose
inequality" has been established by H. Bray in \cite{Bray:2001} and by
G. Huisken and T. Ilmanen in \cite{Huisken-Ilmanen:2001}. We emphasize
that the outermost minimal surface is known to be outer area-minimizing (in particular, it is strongly stable). This variational
feature is of essential importance in both available proofs of the
Penrose inequality. A deep relation between the existence of stable
minimal surfaces in initial data sets and their ADM-mass has been
recognized and exploited by R. Schoen and S.-T. Yau in
\cite{Schoen-Yau:1979-pmt1} in their proof of the positive energy
theorem. Their work has made a profound connection between the
physical concept of mass and the geometry of manifolds with
non-negative scalar curvature.

It is natural to ask if other physical properties of the spacetime
development of an initial data set $(M, g)$ are captured by its
geometry. Maybe they are witnessed by the existence and behavior of
special surfaces in $(M, g)$, and their behavior? The variational
properties associated with constant mean curvature surfaces in $(M,
g)$ generalize the geometric properties of the horizon in a natural
way.

In \cite{Huisken-Yau:1996}, G. Huisken and S.-T. Yau showed
that the asymptotic region of an initial data set $(M, g)$ that is
$\C^4$-asymptotic to Schwarzschild of mass $m>0$ in the sense of
Definition \ref{def:initial_data_sets} is foliated by strictly volume
--preserving stable constant mean curvatures spheres that are
perturbations of large coordinate balls.  Moreover, these spheres are
unique among volume-preserving stable constant mean curvature
spheres in the asymptotic region that lie outside a coordinate ball of
radius $\mc^{-q}$, where $\mc$ denotes their constant mean curvature
and where $q \in (\frac{1}{2}, 1]$. They also concluded that as the
enclosed volume gets larger, these surfaces become closer and closer
to round spheres whose centers converge in the limit as the volume
tends to infinity to the Huisken-Yau ``geometric center of mass" of
$(M, g)$. See also the announcement
\cite[p. 14]{Christodoulou-Yau:1988}. R. Ye \cite{Ye:1996b} has an alternative approach to proving
existence of such foliations. In \cite{Qing-Tian:2007}, J. Qing and
G. Tian strengthened the uniqueness result of \cite{Huisken-Yau:1996}
by showing the following: every volume-preserving stable constant mean
curvature sphere in an initial data set that is $\C^4$-asymptotic to
Schwarzschild of mass $m>0$ that contains a certain large coordinate
ball (independent of the mean curvature of the surface) belongs to
this foliation.
 
The assumption $m>0$ in the results described in the preceding
paragraph is necessary: the constant mean curvature surfaces of $\R^3$
are neither strictly volume-preserving stable nor unique. In view of
the results in \cite{Huisken-Yau:1996, Qing-Tian:2007}, and loosely
speaking, positive mass has the property that it centers large,
outlying volume-preserving stable constant mean curvature
surfaces. Various extensions of these results that allow for weaker
asymptotic conditions have been proven in \cite{metzger:2007},
\cite{huang:2010}, and \cite{Ma:2011, Ma:2012}. In \cite{huang:2009},
L.-H. Huang has shown that the ``geometric center of mass" of
G. Huisken and S.-T. Yau coincides with other invariantly defined
notions for the center of mass.

In his thesis \cite{Bray:1998}, H. Bray started a systematic
investigation of isoperimetric surfaces in initial data sets and their
relationship with mass, quasi-local mass, and the Penrose
inequality. He showed that the isoperimetric surfaces of Schwarzschild
are exactly round centered spheres. He deduced that the large
isoperimetric surfaces in initial data sets that are compact
perturbations of the exact Schwarzschild metric are also round
centered spheres. Furthermore, he gave a proof of the Penrose
inequality under the additional assumption that there exist connected
isoperimetric surfaces enclosing any given volume in $(M, g)$. This
proof builds on H. Bray's important observation that his isoperimetric
Hawking mass is monotone increasing with the volume in this case. (In
fact, H. Bray pointed out that the Hawking mass is monotone along
foliations through connected volume-preserving stable constant mean
curvature spheres whose area is increasing, such as those constructed
in \cite{Huisken-Yau:1996, Ye:1996b}.) In \cite[p. 44]{Bray:1998},
H. Bray conjectured that the volume-preserving stable constant mean
curvatures surfaces of \cite{Huisken-Yau:1996, Ye:1996b} are
isoperimetric surfaces. The results in the present paper confirm this.

\begin{theorem} \label{thm:main} Let $(M, g)$ be an initial data set that is $\C^0$-asymptotic to Schwarzschild of mass $m>0$ in the sense of Definition \ref{def:initial_data_sets}. There exists $V_0 >0$ such that for every $V \geq V_0$ the infimum in 
  \begin{equation*}
    \begin{split}
      A_g(V) :=  \inf \{ \H^2_g(\del^* \Omega) &: \Omega \text { is a
        Borel set of volume } V \text{ that contains} \\
      &\qquad\text{the horizon and has finite perimeter} \}
    \end{split}
  \end{equation*}
is achieved. Every minimizer has a smooth bounded representative whose boundary consists of the horizon and a connected surface that is close to a centered coordinate sphere.   
\end{theorem}

In conjunction with \cite{Huisken-Yau:1996}, we immediately obtain the following corollary. 

\begin{corollary} \label{cor:main} If the initial data set $(M, g)$ is
  $\C^4$-asymptotic to Schwarzschild of mass $m>0$ in the sense of
  Definition \ref{def:initial_data_sets}, then the boundaries of the
  large isoperimetric regions of Theorem \ref{thm:main} coincide with
  the volume-preserving stable constant mean curvature surfaces
  constructed in \cite{Huisken-Yau:1996}. In particular, for every
  sufficiently large volume there exists a unique isoperimetric region
  in $(M, g)$ of that volume. The boundaries of these regions foliate
  the complement of a bounded subset of $(M, g)$.
\end{corollary}

It follows that the isoperimetric profile $A_g(V)$ of $(M, g)$ for
large volumes $V$ is exactly determined. This mirrors the situation in
compact Riemannian manifolds whose scalar curvature assumes its
maximum at a unique point $p$. Under those assumptions, small
isoperimetric regions are known to be perturbations of geodesic balls
centered at $p$. (This follows from \cite{Druet:2002}. See also
\cite[Theorem 2.2]{Johnson-Morgan:2000} and \cite[Corollary
3.12]{Nardulli:2009}.)

G. Huisken has initiated a program where the mass of an initial data
set and the quasi-local mass of subsets of initial data sets are
studied via isoperimetric deficits from Euclidean space. One great
advantage of this approach is that only very low regularity is
required of the initial data set. Theorem \ref{thm:main} identifies
$m$ as the only sensible candidate for any notion of mass that is
defined in terms of $A_g(V)$ when the initial data set is
$\C^0$-asymptotic to Schwarzschild of mass $m>0$,
cf. \cite{Bray:1998}. A result of X.-Q. Fan, Y. Shi, and L.-T. Tam
\cite[Corollary 2.3]{Fan-Shi-Tam:2009} subsequent to the work of
G. Huisken shows that the ADM mass of an initial data set that has
integrable scalar curvature and which is $\C^0$-asymptotic to
Schwarzschild of mass $m>0$ equals $m$.

In a sequel \cite{Eichmair-Metzger:2012b} to this paper, we generalize
our main result Theorem \ref{thm:main} to arbitrary dimensions. We
also show that in Corollary \ref{cor:main}, it is enough to assume
that $(M, g)$ is $\C^2$-asymptotic to Schwarzschild of mass $m>0$. In
Appendix H of \cite{Eichmair-Metzger:2012b} we provide an extensive
overview of the portion of the literature on isoperimetric regions on
Riemannian manifolds related to our results.

\subsection*{Structure of this paper} In Section
\ref{sec:definitions-notation} we introduce the precise decay
assumptions for initial data sets that we use in this paper, and we
define what exactly we mean by isoperimetric and locally isoperimetric
regions. In Section \ref{sec:volume_comparison} we prove an effective
volume comparison theorem for regions in initial data sets that are
$\C^0$-asymptotic to Schwarzschild. In Section \ref{sec:regularity} we
review the classical results on the regularity of isoperimetric
regions and behavior of minimizing sequences for the isoperimetric
problem that we need in this paper. The effective volume comparison
theorem is applied in Section \ref{sec:center} to show that
isoperimetric regions exist for every sufficiently large volume in
initial data sets that are $\C^0$-asymptotic to Schwarzschild, and
that these regions become close to large centered coordinate balls as
their volume increases. In Section \ref{sec:general_asymptotics} we
present our most general result on the behavior of isoperimetric
regions in asymptotically flat initial data sets that are not assumed
to be close to Schwarzschild: either such regions slide away entirely
into the asymptotically flat end of the initial data set as their
volume grows large, or they begin to fill up the whole initial data
set. The results in this section are largely independent of the
remainder of the paper. In Appendix \ref{sec:integral_decay_estimates}
we collect several useful lemmas regarding integrals of polynomially
decaying quantities over surfaces with quadratic area growth. In
Appendix \ref{sec:Bray_thesis} we summarize some steps and results
from H. Bray's thesis. Appendix \ref{sec:minimizing_limit} contains a
``friendly" proof that limits of isoperimetric regions with divergent
volumes in initial data sets have area-minimizing boundaries. This
fact is used in the proof of Theorem \ref{thm:PMT}.

\subsection*{Acknowledgements} We have had helpful and enjoyable
conversations with L. Rosales and B. White regarding the geometric
measure theory used in this paper. We are grateful to H. Bray and
S. Brendle for their interest, feedback, and enthusiasm and to
G. Huisken for his great encouragement, and for sharing with us his
perspective on isoperimetric mass. We are grateful for the referees
for their careful reading, their useful suggestions, and for pointing
out to us Corollary 2.3 in \cite{Fan-Shi-Tam:2009}.

%%%%%%%%%%%%%%%%%%%%%%%%%%%%%%%%%%%%%%%%%%%%%%%%%%%%%%%%%%%%%%%%%%%%%%%%%%%%%%%%%%%%%%%
%%%%%%%%%%%%%%%%%%%%%%%%%%%%%%%%%%%%%%%%%%%%%%%%%%%%%%%%%%%%%%%%%%%%%%%%%%%%%%%%%%%%%%%
%%%%%%%%%%%%%%%%%%%%%%%%%%%%%%%%%%%%%%%%%%%%%%%%%%%%%%%%%%%%%%%%%%%%%%%%%%%%%%%%%%%%%%%

\section{Definitions and notation} \label{sec:definitions-notation}

\begin{definition} \label{def:Schwarzschild} Let $m>0$. We denote by
  $(M_m, g_m)$ the complete Riemannian manifold $(\R^3\setminus \{0\},
  \phi_m^4\sum_{i=1}^3 dx_i^2)$, where $\phi_m = \phi_m(x) := 1 +
  \frac{m}{2r}$, $r = r(x) := \sqrt{x_1^2 + x_2^2 + x_3^2}$, and where
  $(x_1, x_2, x_3)$ are the coordinate functions on $\R^3$. $(M_m,
  g_m)$ is a totally geodesic spacelike slice of the Schwarzschild
  spacetime of mass $m>0$. We refer to $(M_m, g_m)$ as the
  Schwarzschild metric of mass $m>0$ for brevity, to the coordinates
  $(x_1, x_2, x_3)$ as isotropic coordinates on $(M_m, g_m)$, and to
  $r(x)$ as the isotropic radius of $x \in M_m$.
\end{definition}
The conformal factor $\phi_m$ is harmonic on $\R^3 \setminus
\{0\}$. It follows that the scalar curvature of $g_m$ vanishes. The
coordinate spheres $\{ x \in M_m : r(x) = r\} \subset M_m$ will be
denoted by $S_r$. Note that $S_{\frac{m}{2}}$ is a minimal surface. It
is called the horizon of $(M_m, g_m)$. The inversion $x \to \left(
  \frac{m}{2} \right)^2 \frac{x}{r(x)^2} $ induces a reflection
symmetry of $(M_m, g_m)$ across the horizon. The area of the isotropic
coordinate sphere $S_r$ is equal to $\phi_m^4 4 \pi r^2 $. Its mean
curvature with respect to the unit normal $\phi_m^{-2} \partial_r$
equals $\phi_m^{-3} (1 - \frac{m}{2r}) \frac{2}{r}$. The Hawking mass
$m(\Sigma) := (16 \pi)^{-3/2} \sqrt{\H^2_{g_m}(\Sigma)} \left(16 \pi -
  \int_\Sigma \mc^2_\Sigma d \H^2_{g_m} \right)$ which is defined on
closed surfaces $\Sigma \subset M_m$ is equal to $m$ when $\Sigma =
S_r$.

\begin{definition} \label{def:initial_data_sets} An initial data set
  $(M, g)$ is a connected complete Riemannian $3$-manifold, possibly
  with compact boundary, such that there exists a bounded open set $U
  \subset M$ with $M \setminus U \cong_x \R^3 \setminus B(0,
  \frac{1}{2})$ and such that in the coordinates induced by $x = (x_1,
  x_2, x_3)$,
\begin{eqnarray*} 
r |g_{ij} - \delta_{ij}| + r^2 |\partial_k g_{ij}| + r^3 |\partial^2_{k l} g_{ij}| \leq C \text{ where }  r:= \sqrt {x_1^2 + x_2^2 + x_3^2}.
\end{eqnarray*}  
If $\partial M \neq \emptyset$, we assume that $\partial M$ is a
minimal surface, and that there are no compact minimal surfaces in $M$
besides the components of $\partial M$.  The boundary of $M$ is called
the horizon of $(M, g)$.  Given $m > 0$ and an integer $k\geq0$, we
say that an initial data set is $\C^k$-asymptotic to Schwarzschild of
mass $m>0$ if
\begin{eqnarray} \label{eqn:decaygm}
\sum_{l=0}^k  r^{2 + l}|\partial^l(g - g_m)_{ij}| \leq C \text{ where } (g_m)_{ij} = (1 + \frac{m}{2r })^4 \delta_{ij}. 
\end{eqnarray} 
\end{definition}

A few remarks are in order. The decay assumptions for initial data
sets here are quite weak. In particular, the ADM-mass is not defined
for such initial data sets unless a further condition, namely the 
integrability of the scalar curvature, is imposed. 

We extend $r$ as a smooth regular function to the entire initial data
set $(M, g)$ such that $r(U) \subset [0, 1)$, except for the case of
exact Schwarzschild $(M_m, g_m)$, where we retain the convention that
$r(x)$ denotes the isotropic radius introduced just below
Definition \ref{def:Schwarzschild}.  We use $S_r$ to denote the
surface $\{ x \in M : |x| = r\}$, and $B_r$ to denote the region $\{ x
\in M : |x| \leq r\}$. We will refer to $S_r$ as the centered
coordinate sphere of radius $r$. We will not distinguish between the
end $M \setminus U$ of $M$ and its image $\R^3 \setminus
B(0,\frac{1}{2})$ under $x$. By  the work of W. Meeks, L. Simon, and
S.-T. Yau \cite{Meeks-Simon-Yau:1982} (see also the discussion in
\cite[Section 4]{Huisken-Ilmanen:2001}), $M$ is diffeomorphic to
$\R^3$ minus a finite number of open balls whose closures are
disjoint.

Given an initial data set $(M, g)$, we fix a complete Riemannian
manifold $(\hat M, \hat g)$ diffeomorphic to $\R^3$ that contains $(M,
g)$ isometrically. We say that a Borel set $U \subset \hat M$ contains
the horizon if $\hat M \setminus M \subset U$. If such a set $U$ has
locally finite perimeter, we denote its reduced boundary in $(\hat M,
\hat g)$ by $\partial^* U$. Note that $\partial^* U$ is supported in
$M$, and that $\H^2_g (\partial^* U) = \H^2_{\hat g} (\partial^*
U)$. To lighten the notation, we write $\CL^3_g(U) := \CL^3_{\hat g} ( U
\cap M) $ for short.
 
\begin{definition} \label{defn:isoprofile} The isoperimetric area function $A_g : [0, \infty) \to [0, \infty)$ is defined by 
  \begin{equation*} 
    \begin{split}
      A_g(V) :=  \inf \{\H^2_g(\partial^* U) &: U \subset \hat M
      \text{ is a Borel containing the horizon} \\
      &\quad \text{and of finite perimeter with } \CL^3_g(U) = V \}. 
    \end{split}
  \end{equation*} 
  A Borel set $\Omega \subset \hat M$ containing the horizon and of
  finite perimeter such that $\CL^3_{g} (\Omega) = V $ and $A_g(V) =
  \H^2_{g} (\partial^* \Omega)$ is called an isoperimetric region of
  $(M, g)$ of volume $V$. A Borel set $\Omega \subset \hat M$
  containing the horizon and of locally finite perimeter is called
  locally isoperimetric if $\H^2_{g} (B \cap \partial^* \Omega) \leq
  \H^2_{g} (B \cap \partial^* U)$ whenever $B \subset \hat M$ is a
  bounded open subset of $\hat M$ and $U \subset \hat M$ is a Borel
  set containing the horizon and of locally finite perimeter such that
  $\CL^3_{\hat g} (\Omega \cap B) = \CL^3_{\hat g} (U \cap B)$ and
  $\Omega \Delta U \Subset B$.
\end{definition}    

The definition of $A_g$ as well as that of isoperimetric and locally
isoperimetric regions is independent of the particular extension
$(\hat M, \hat g)$ of $(M, g)$. Note that $A_g(0) = \H^2_g(\partial
M)$ and that $A_g(V) > \H^2_g(\partial M)$ for every $V > 0$. The
latter assertion follows from the assumption that the boundary of $M$
is an outermost minimal surface. Locally isoperimetric regions arise
naturally as limits of isoperimetric regions whose volumes diverge. A
good example to keep in mind is a half-space in $\R^3$.  Standard
results in geometric measure theory imply that the boundary of a
(locally) isoperimetric region $\Omega$ is smooth, that $ \Omega
\cap \partial M = \emptyset$ unless the enclosed volume
$\CL^3_g(\Omega) = \CL^3_{\hat g} (\Omega \cap M)$ is $0$, and that
isoperimetric regions are compact. Indications of the proofs of these
facts with precise references to the literature to assist the reader
are given in Section \ref{sec:regularity} below.

The inequalities in the following lemma are well-known, and we recall
them for convenient reference.

\begin{lemma} \label{lem: crude isoperimetric inequality}
Let $(M, g)$ be an initial data set. There exists a constant $\gamma > 0$ such that 
\begin{eqnarray} \label{eqn:relativeSobolev} 
\left( \int_{M} |f|^{\frac{3}{2}} d \CL^3_{g}\right)^{\frac{2}{3}} \leq \gamma \int_{M} |\nabla f| d \CL^3_{g} \text{ for every } f \in \C_c^1(M).
\end{eqnarray}
If the boundary of $M$ is empty, the constant $\gamma >0$ can be chosen such that for any bounded Borel set $\Omega \subset M$ with finite perimeter one has that $$\CL^3_{g} (\Omega)^{\frac{2}{3}} \leq \gamma \H^2_{g}(\partial^* \Omega).$$   
\begin{proof} The Sobolev inequality stated here can be obtained exactly as in \cite[Lemma 3.1]{Schoen-Yau:1979-pmt1} by combining, in a contradiction argument, the Euclidean Sobolev inequality in the form
\begin{equation*}
\left( \int_{\R^3 \setminus B(0, 1)} |f|^\frac{3}{2} d \CL^3_\delta \right)^{\frac{2}{3}} \leq \gamma_0 \int_{\R^3 \setminus B(0, 1)} |\nabla f| d \CL^3_\delta \text{ for all } f \in \C^1_c(\R^3)
\end{equation*}
and Poincar\'e--type inequalities (see
\cite[$\mathsection$8.12]{Lieb-Loss:2001} for the appropriate version
with critical exponent) on precompact coordinate charts. We recall
(cf. \cite[Theorem II.2.1]{Chavel:2001}) that the isoperimetric
estimate for smoothly bounded compact regions $\Omega$ follows from
applying this Sobolev inequality to approximations of the indicator
function $\chi_\Omega$ by Lipschitz functions that are one on $\Omega$
and that drop off to $0$ linearly in the distance from $\Omega$. The
isoperimetric inequality for sets of finite perimeter is obtained by
approximation through smooth sets.
\end{proof}
\end{lemma}

%%%%%%%%%%%%%%%%%%%%%%%%%%%%%%%%%%%%%%%%%%%%%%%%%%%%%%%%%%%%%%%%%%%%%%%%%%%%%%%%%%%%%%%
%%%%%%%%%%%%%%%%%%%%%%%%%%%%%%%%%%%%%%%%%%%%%%%%%%%%%%%%%%%%%%%%%%%%%%%%%%%%%%%%%%%%%%%
%%%%%%%%%%%%%%%%%%%%%%%%%%%%%%%%%%%%%%%%%%%%%%%%%%%%%%%%%%%%%%%%%%%%%%%%%%%%%%%%%%%%%%%

\section{Effective refinement of H. Bray's characterization of isoperimetric surfaces in Schwarzschild} \label{sec:volume_comparison}

In his thesis \cite{Bray:1998}, H. Bray proved that large isoperimetric
surfaces of compact perturbations of the Schwarzschild metric with
mass $m>0$ are centered coordinate spheres in isotropic
coordinates. In this section, we refine H. Bray's work to derive an
effective lower bound for the isoperimetric defect of off-centered
surfaces in Schwarzschild. This bound gives us enough
quantitative information to characterize large isoperimetric surfaces
in manifolds that are $\C^0$-asymptotic to Schwarzschild of mass
$m>0$, as we will see in Section \ref{sec:center}.

We begin with a description of the ``volume-preserving" charts used by H. Bray. We refer the reader to Appendix \ref{sec:Bray_thesis} for an overview of related results from H. Bray's thesis that should be noted in this context. 

Let $\alpha >0$. Consider the metric cone $\alpha^{-2} d s^2 + \alpha s^2 g_{\S^2}$ on
$(0, \infty) \times \S^2$. The
sphere $\{c\} \times \mathbb{S}^2$ has area $\alpha 4 \pi c^2$ and
mean curvature $\frac{2 \alpha}{c}$. One can choose $c>0$ and
$\alpha>0$ so that the intrinsic geometry and (constant, outward) mean
curvature of the sphere $\{c\} \times \S^2$ with respect to this cone coincide with
that of the centered sphere $S_r$ (with $r > \frac{m}{2}$) in $(M_m,
g_m)$. Using the remarks below Definition \ref{def:Schwarzschild}, we
see that this requires that 
\begin{eqnarray*}
c^3 &=& r^3 \frac{\phi_m^7}{1 - m/(2r)} = r^3 (1 + \frac{4m}{r}+ O(\frac{1}{r^2})),  \\
\alpha &=& \phi_m^{-\frac{2}{3}} (1 - \frac{m}{2r})^{\frac{2}{3}} = 1 - \frac{2m}{3r} + O(\frac{1}{r^2}). 
\end{eqnarray*}
 Note that $\alpha \in (0, 1)$ and
that $\alpha \nearrow 1$ as $r \to \infty$. We emphasize that $\alpha$
and $c$ are uniquely determined by $r$. The scalar curvature of
this conical metric equals $2 \frac{1 - \alpha^3}{\alpha s^2}$. In
particular, it is positive for $\alpha \in (0, 1)$.

The volume between the sphere $S_r$ of (isotropic) radius $r$ and
the horizon $S_{\frac{m}{2}}$ in Schwarzschild is $4 \pi \int_{\frac{m}{2}}^r (1 +
\frac{m}{2 r})^6 r^2 dr = \frac{4 \pi r^3}{3} (1 + \frac{9 m }{2 r} +
O(\frac{1}{r^2}))$. The volume of the (punctured) disk $(0, c] \times
\S^2$ in the cone metric above equals $\frac{4 \pi c^3}{3} = \frac{4
  \pi r^3}{3} (1 + \frac{4m}{r}+ O(\frac{1}{r^2}))$. We denote the difference between the
Schwarzschild volume and the cone volume by $V_0$. Note that $V_0 =
\frac{4 \pi r^3}{3} \frac{m}{2r} + O(r) = \frac{4 \pi c^3}{3}
\frac{m}{2c} + O(c)$.

Following H. Bray, we represent the part of the Schwarzschild
metric $(M_m, g_m)$ that lies outside the centered sphere of isotropic
radius $r$ in the form $u_c^{-2} ds^2 + u_c s^2 g_{\S^2}$ on $[c,
\infty) \times \S^2$ for some radial function $u_c$. This requires
that $u_c(c) = \alpha$ and $\partial u_c |_c = 0$, and that $u_c$
satisfies a certain second order ordinary differential equation (to make the scalar curvature
vanish). We remark that by Birkhoff's theorem and the constancy of
the Hawking mass along centered spheres in Schwarzschild there is a
first integral for $u_c$.

Finally, let $g_m^c := u_c^{-2} ds^2 + u_c s^2 g_{\S^2}$ be the metric
on $(0, \infty) \times \S^2$ with $u_c(s) =
\alpha$ for $s \leq c$ and $u_c(s)$ is equal to $u_c = u_c(s)$ from
the preceding paragraph when $s \geq c$. To summarize, we have that
$u_c$ is $\C^{1, 1}$, is radial, and is such that the set $[c, \infty)  \times \mathbb{S}^2$ in the $g_m^c$ metric is isometric to
the exterior of a round sphere $S_r$ of isotropic radius $r$ in the
Schwarzschild manifold of mass
$m$, and such that $u_c(s) = \alpha$ for $s \leq c$ for some constant
$\alpha$, such that the boundaries $\{c\} \times \mathbb{S}^2$ and $S_c$ correspond and such that the mean curvature of
$\{c\} \times \mathbb{S}^2$ from the inside (the
conical part) matches that from the outside (in Schwarzschild).

A key feature of this construction used by H. Bray is that the volume element $s^2 ds \wedge
dg_{\S^2}$ of
$g_m^c$ is independent of $c$.
By definition of $V_0$, the Schwarzschild volume between
the horizon and a centered Schwarzschild sphere isometric to the
sphere $\{s\} \times \S^2$ (with $s \geq c$) in $((0, \infty) \times \mathbb{S}^2,
g_{m}^c)$ equals $\frac{4 \pi s^3}{3} + V_0$. Thus its area equals
$A_m( \frac{4 \pi s^3}{3}+V_0)$, where $A_m$ is the function that
assigns to every volume (measured relative to the horizon) the area of
a centered sphere in Schwarzschild that encloses that volume. On the other hand, the
area of $\{s\} \times \S^2$ is given explicitly by $u_c(s) 4\pi s^2
$. In combination this yields the following explicit expression
for $u_c$:
\begin{equation*}
  u_c (s)
  :=
  \frac{A_m(V+V_0)}{(36 \pi)^{\frac{1}{3}} V^{\frac{2}{3}}}
  \nonumber \text{ for all } s \geq c, \text{ where }  V := \frac{4 \pi
s^3}{3}; \text{ cf.  \cite[p. 34]{Bray:1998}}.
\end{equation*}
It is known (and easy to verify) that
\begin{eqnarray*}
  \frac{\H^2_{g_m}(\partial B(0, r))}{(36 \pi)^{\frac{1}{3}} \CL^3_{g_m}( B(0, r)
    \setminus B(0, \frac{m}{2}))^{\frac{2}{3}}}
  =
  1 - \frac{m}{r} + O\left( \frac{1}{r^2} \right)
\end{eqnarray*}
and from this that
\begin{equation} \label{eqn:expansionprofile}
  A_m(V) = 4 \pi R^2 \left( 1 - \frac{m}{R} +
O\left(\frac{1}{R^2}\right)\right)
  \text{ where }
  R := \left(\frac{3 V}{4 \pi} \right)^{\frac{1}{3}}.
\end{equation}
By assumption we have that $u_c(c) = \alpha = 1 - \frac{2 m }{3c} +
O(\frac{1}{c^2})$. For a fixed $\tau \in (1, \infty)$ we are
interested in estimating $u_c ( \tau c) - u_c (c)$. Note that
\begin{eqnarray*}
    u_c (\tau c )
    &=&
    \frac{A_m\left(\frac{4 \pi (\tau c)^3}{3} \left( 1 + \frac{m}{2
         \tau^3 c} + O( \frac{1}{c^2} ) \right) \right)}{4 \pi(\tau c)^2} \\
    &=&
    \left(1 + \frac{m }{2 \tau^3 c} + O\left(\frac{1}{c^2}\right)\right)^{\frac{2}{3}} \left(1 -
    \frac{m}{\tau c} + O\left(\frac{1}{c^2}\right)\right)
    \\[1ex]
    &=&
    1 - \frac{2m}{3c} \left(\frac{3}{2\tau} - \frac{1}{2\tau^3}\right) + O\left(\frac{1}{c^2}\right).
\end{eqnarray*}
This means that for $\tau_0 \in (1, \infty)$ fixed and $\tau \geq \tau_0$ we have that 
\begin{eqnarray} \label{eqn:estimateu} 
u_c (\tau c ) - u_c (c) = u_c(\tau c) - \alpha \geq \frac{1}{2} \frac{(\tau+ \frac{1}{2}) (\tau - 1)^2}{\tau^3}  \frac{2m}{3c} 
\end{eqnarray}
provided that $c$ is sufficiently large (depending only on $m$ and $\tau_0$). This quantifies the fact from \cite{Bray:1998} that $u_c(s)$ is increasing for $s \geq c$; see Appendix \ref{sec:Bray_thesis}.

In the proof of the following lemma, we supply some additional details and in fact make a slightly different claim than \cite[p. 37]{Bray:1998}:
  
\begin{lemma}  [Cf. \protect{\cite[p. 37]{Bray:1998}}] \label{lem: construction of w} Consider the conical part of the
metric $g_{m}^c$ given
  by $\alpha^{-2} ds^2 + \alpha s^2 g_{\S^2}$ on $(0, c) \times \S^2$
  where $\alpha$ and $c$ are such that the outward mean curvature of
  $\{c\} \times \S^2$ with respect to $g_{m}^c$ is the same as that of
  a centered sphere $S_r$ of area $\alpha 4 \pi c^2$ in Schwarzschild
  with mass $m$. Then there exists $s_0 \geq 0$ and a smooth
  radial function $w_c : (s_0, c] \to [1, \infty)$  such
  that $w_c^4 g_m^c$ is isometric to the Schwarzschild metric interior
  to the mean-convex sphere $S_r$, and such that $w_c(c) = 1$ and
  $\partial_s w_c|_c = 0$.
\begin{proof}
  The scalar curvature  $\Scal_{g_m^c} = 2 \frac{1 - \alpha^3}{ \alpha s^2}$ of the conical part of the metric $g_m^c$ is
  strictly positive. For the conformal metric $w_c^4 g_{m}^c$ to be
  isometric to (part) of a Schwarzschild metric, it is necessary that its
  scalar curvature vanishes and hence that $w_c$ is a solution of the
  elliptic (Yamabe) equation $- 8 \Delta_{g_m^c} w_c + \Scal_{g_m^c} w_c =
  0$. This equation reduces to a second order ordinary differential equation if we are solving
  for radial functions. Hence we can solve this equation for $s$ close to $c$ with initial
  data $w_c(c)= 1$ and $\partial_s w_c|_c =0$. By
  Birkhoff's theorem, $w_c^4 g_m^c$ is isometric to (part of) a
  Schwarzschild metric. To determine the mass $\hat m$ of this metric,
  we evaluate its Hawking mass on the sphere $\{c\} \times
  \S^2$. Since the initial data are chosen so that the area and mean
  curvature of this sphere coincide with that of an umbilic constant mean curvature sphere
  of a Schwarzschild metric of mass $m$, we obtain that $\hat m =
  m$. On every connected open sub-interval of $(0, c]$ that
  contains $c$ and on which the solution $w_c$ exists and is
  non-negative, we have that $\Delta_{g_m^c} w_c =
  \frac{1}{s^2} \partial_s (s^2 \alpha^2 \partial_s w_c) = \frac{1}{8}
  R_{g_m^c} w_c \geq 0$. Integrating up and using that $\partial_s w_c|_c
  = 0$, it follows that $\partial_s w_c \leq 0$ on any such interval. Moreover, we see that $w_c(s)$ is a decreasing function of $s$. In particular,
  $w_c \geq 1$ on any such interval. The constancy of the Hawking mass
  is equivalent to the existence of a first integral for the ordinary differential equation satisfied by
  $w_c$.  We let $(s_0, c]$ be the maximally left-extended interval of
  existence of the solution $w_c$. Since the metric $w_c^4 g_m^c$ on $(s_0, c]
  \times \S^2$ is isometric to (part) of a Schwarzschild metric, it
  follows that $w_c \nearrow \infty$ as $s \searrow s_0$ and that we actually
  obtain an isometric copy of the full spatial Schwarzschild metric that lies
  to the mean-concave side of $S_r$.
\end{proof}
\end{lemma}

Fix an isotropic sphere $S_r$ in $(M_m, g_m)$, let $g_m^c$ be the metric on
$(0, \infty) \times \mathbb{S}^2$ constructed above, and let $w_c$ be as in Lemma \ref{lem:
construction of w}, extended by $1$ to $s \geq c$, so that $((s_0, \infty) \times \mathbb{S}^2, w_c^4g_m^c)$ is
isometric to $(M_m, g_m)$. We will refer to it as the volume-preserving chart associated with $S_r$. Recall that the isotropic sphere $S_r$ corresponds to
the coordinate sphere $\{c\} \times \mathbb{S}^2$ in $((s_0, \infty) \times \mathbb{S}^2,
w_c^4g_m^c)$.  Finally, let $\Sigma$ be a surface in $((s_0, \infty) \times \mathbb{S}^2, w_c^4g_m^c)$ homologous to the horizon that encloses the same (relative) volume as $\{c\} \times \mathbb{S}^2$. The reader should keep in mind that $\Sigma$ might consist of the horizon itself (enclosing volume zero) and another surface that is the boundary of a compact set that is disjoint from the horizon. In this case the area of the horizon is counted as part of the area of $\Sigma$.  

Since $w_c \geq 1$ it follows that the volume enclosed by $\Sigma$ with respect to the $g_m^c$ metric (and relative to the horizon of the Schwarzschild metric
$w_c^4 g_m^c$ in the same coordinate chart) is at least that enclosed by $\{c\} \times \mathbb{S}^2$.
Note that as quadratic forms, $\alpha^2 g_m^c \leq \delta := ds^2 + s^2 g_{\mathbb{S}^2} \leq u_c^{-1} g_{m}^c$, since $\alpha \leq u_c \leq 1$. 
As in \cite{Bray:1998}, but meticulously recording the error
terms in the computation, we obtain that
\begin{align}  
\H^2_{g_m} (\Sigma) =  \H^2_{ w_c^4 g_m^c} (\Sigma)  &\geq   \H^2_{ g_m^c} (\Sigma) \nonumber
  = \int_\Sigma d \H^2_{g_m^c} \\
  &\geq \int_\Sigma u_c d \H^2_{\delta} \nonumber
  = \int_{\Sigma} (u_c - \alpha) d \H^2_{\delta} + \alpha \int_{\Sigma} d \H^2_{\delta}
  \\
  &\geq \int_{\Sigma} (u_c - \alpha) d \H^2_{\delta} + \alpha \H^2_{\delta} (\{c\} \times \mathbb{S}^2) \nonumber
    \\
  &= \int_\Sigma (u_c - \alpha) d \H^2_{\delta} + \H^2_{g_m} (S_r)  \nonumber \\
    &\geq \alpha^2  \int_\Sigma (u_c - \alpha) d \H^2_{g_m^c} + \H^2_{g_m} (S_r)  \label{eqn: refinement of Bray's estimate}
\end{align}
The third inequality follows from the Euclidean isoperimetric inequality. The definition below is natural in view of this estimate and the expansion of $u_c$ for $s \geq c$. 

\begin{definition}
  \label{def: off-center} Let $(M, g)$ be an initial data set that is $\C^0$-asymptotic to Schwarzschild of mass $m >0$. Let $\Omega$ be a bounded Borel set with finite perimeter in $(M, g)$ that contains the horizon. Given parameters $\tau >1$ and $\eta \in (0, 1)$ we say that such a set $\Omega$ is $(\tau, \eta)$-off-center if
  \begin{enumerate}
  \item $\CL^3_g (\Omega)$ is so large that there exists a coordinate sphere $S_r = \partial B_r$ with $\CL^3_g(\Omega) = \CL^3_g(B_r)$ and $r \geq 1$, and if
  \item $\H^2_g (\partial^* \Omega \setminus B_{\tau r}) \geq \eta \H_g^2 (S_r)$.
  \end{enumerate}
\end{definition}

\begin{figure}
  \centering
  \resizebox{0.5\linewidth}{!}{\input{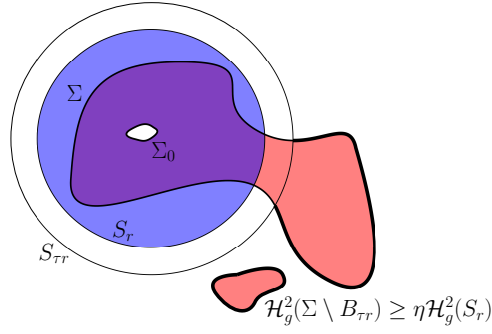}}
  \caption{A large portion of the area of $\Sigma$ lies outside of
    $B_{\tau r}$}
\end{figure}

Let $\Sigma$ be a surface in Schwarzschild containing the horizon and enclosing volume $V$ with it, and let $r \geq \frac{m}{2}$ be such that $\CL^3_{{g_m}} (B_r \setminus B_{\frac{m}{2}}) = V$. Assume that $r$ is large and that $\Sigma$ is $(\tau, \eta)$-off-center. It is easy to see that as $c \to \infty$, the isotropic sphere $S_{\tau r}$ corresponds to $ \{(\tau + o(1)) c\} \times \mathbb{S}^2$ in the volume-preserving chart. 
It follows that for $r$ sufficiently large a portion $\eta \H^2_{g_m}(S_r)$ of the area of $\Sigma$ lies in the region $(\frac{1+\tau}{2} c, \infty) \times \mathbb{S}^2$ of the volume-preserving chart. We can use this
information together with (\ref{eqn:estimateu}), replacing $\tau$ by $(1 + \tau)/2$, to continue the estimate  (\ref{eqn: refinement of Bray's estimate}). We obtain that
\begin{eqnarray} \nonumber
  \H^2_{g_m} (\Sigma)
  &\geq&
  \H^2_{g_m} (S_r) + \frac{ \eta m } {96} \left( 1 -\frac{1}{\tau}  \right)^2
\frac{\H^2_{g_m} (S_r)}{r}  \\ \nonumber
  &\geq&
  \H^2_{g_m} (S_r) + \frac{\eta m \pi}{24}  \left(1 -  \frac{1}{\tau}  \right)^2 r
\end{eqnarray}
for all $r$ sufficiently large,  depending only on $m$ and $\tau$. We have used that $2 \alpha \geq 1$ and $2r \geq c$ for $r$ sufficiently large here, and that $\H^2_{g_m}(S_r) \geq 4 \pi r^2$. 

The arguments leading to this estimate also apply if $\Sigma=\del^* \Omega$ is the reduced boundary of a finite-perimeter Borel set $\Omega$ containing the horizon. 
\begin{proposition} [Effective volume comparison in
  Schwarzschild] \label{prop:effectiveSchwarzschild} For $m>0$ and
  $(\tau, \eta) \in (1, \infty) \times (0, 1)$ there exists $V_0 >0$
  with the following property: Let $V \geq V_0$ and $r \geq
  \frac{m}{2}$ such that $V= \CL^3_{g_m} (B_r \setminus
  B_{\frac{m}{2}})$, and let $\Omega \subset \R^3$ be a bounded finite
  perimeter Borel set such that $B_{\frac{m}{2}} \subset \Omega$ and
  $\CL^3_{g_m}(\Omega \setminus B_{\frac{m}{2}}) = V$. If $\Omega$ is
  $(\tau, \eta)$-off-center, i.e. if $\H^2_{g_m} ( \partial^* \Omega
  \setminus B_{\tau r} ) \geq \eta \mathcal{H}_{g_m} (S_r)$, then
\begin{equation} \label{eqn: comparison in Schwarzschild}
  \H^2_{g_m}(\partial^* \Omega) \geq \H^2_{g_m} (S_r) + \frac{\eta m \pi}{24} \left(1 -  \frac{1}{\tau}\right)^2  r.
\end{equation}
\end{proposition}

This is our effective refinement of H. Bray's argument in exact Schwarzschild. The study of effective isoperimetric inequalities is classical with much recent activity, see e.g. \cite{Fusco-Maggi-Pratelli:2008, Figalli-Maggi-Pratelli:2010, Cicalese-Leonardi:2010}. The effective volume comparison in Proposition \ref{prop:effectiveSchwarzschild} is not obtained from lifting an effective isoperimetric inequality from the Euclidean background. It depends on the particular form of the Schwarzschild metric in an essential way. 

In the proof of the theorem below, we will appreciate that we can quantify how much an off-center surface in Schwarzschild falls short of being isoperimetric. The defect is large enough for us to carry out the comparison on arbitrary initial data sets that are $\C^0$-asymptotic to Schwarzschild of mass $m>0$:
    
\begin{theorem} \label{thm:effective_volume_comparison} Let $(M, g)$ be an initial data set that is $\C^0$-asymptotic to Schwarzschild of mass $m >0$. For every tuple $(\tau, \eta) \in (1, \infty) \times (0, 1)$ and constant $\Theta >0$ there exists a constant $V_0 >0$ with the following property: Let $\Omega$ be a bounded finite-perimeter Borel set containing the horizon with $\CL^3_g(\Omega) \geq V_0$ that is $(\tau, \eta)$-off-center and such that $\H^2_g(\partial^* \Omega) ^{\frac{1}{2}} \CL^3_g(\Omega)^{-\frac{1}{3}} \leq \Theta$ and $\H^2_g(B_\sigma \cap \partial^* \Omega) \leq \Theta \sigma^2$  for all $\sigma \geq 1$. Then 
\begin{eqnarray} \label{mainestimate}
\H^2_g(S_r) + \frac{\eta m \pi}{300}  \left(1 -  \frac{1}{\tau}  \right)^2 r \leq \H^2_g(\partial^* \Omega).
\end{eqnarray} 
Here, $S_r \subset M$ is the centered coordinate sphere that encloses $g$-volume $\CL_g^3(\Omega)$ with the horizon.
\end{theorem}
{\it Remark: } The form of the constant that multiplies $r$ in (\ref{mainestimate}) is immaterial. The explicit expression is given to indicate the dependence on the parameters. 
\begin{proof} 
For ease of exposition, we only consider smooth regions $\Omega$. The result for sets of finite perimeter follows by approximation. By Lemma \ref{lem: crude isoperimetric inequality}, $\H^2_g (\partial \Omega) \to \infty$ as $\CL^3_g(\Omega) \to \infty$. Note also that $\CL^3_g(\Omega)=  \frac{4 \pi r^3}{3} + O(r^2)$. 

We break the argument into several steps:
\begin{enumerate} [(a)]  
\item Let $\tilde \Omega := \Omega \cup B_{1} \subset M$. Let $\tilde \Omega_m :=
\left( x (\Omega \setminus B_{1})  \cup B(0, 1) \right) \setminus B(0, \frac{m}{2})$ be the corresponding region in Schwarzschild.
\item Note that $\CL^3_g(\tilde \Omega) = \CL^3_g(\Omega) + O(1)$ and $\H^2_g(\partial \tilde \Omega) = \H^2_g (\partial \Omega) + O(1)$. Moreover, $\tilde \Omega$ satisfies $\H^2_g (B_\sigma \cap \partial \tilde \Omega) \leq \tilde \Theta \sigma^2$ for all $\sigma \geq 1$ where $\tilde \Theta$ depends only on $\Theta$ and $(M, g)$.  
\item By Corollary \ref{cor:surface_comparison} with $\beta =
  \frac{1}{2}$, 
  \begin{equation*}\H^2_{g_m} (\partial \tilde \Omega_m) \leq \H^2_g(\partial \tilde \Omega) +
O(\H^2_g(\partial \tilde \Omega)^{\frac{1}{4}}) \leq \H^2_g(\partial \Omega)
+ O (\H^2_g (\partial \Omega)^{\frac{1}{4}}).\end{equation*}
\item \label{item4} By Lemma \ref{lem:volume_comparison} with $\alpha = \frac{3}{2}$, $\CL^3_{g_m} (\tilde \Omega_m) = \CL^3_g(\Omega) +
O(\CL_g^3(\Omega)^{\frac{1}{2}})$. 

\item \label{item5} By Lemma \ref{lem:volume_comparison} with $\alpha = \frac{3}{2}$ and choice of $r$, $\CL^3_{g_m} (B_r \setminus B_{\frac{m}{2}}) =  \CL^3_{g_m} (B_r \setminus B_1) + O(1) = \CL^3_g (B_r \setminus B_1) +
O(\CL^3_g(B_r \setminus B_1)^{\frac{1}{2}}) = \CL^3_g (\Omega) + O(\CL^3_g(\Omega)^{\frac{1}{2}})  $. 

\item By (\ref{item4}) and (\ref{item5}) and choice of $r$ we have that $\CL^3_{g_m} (\tilde \Omega_m) = \CL^3_{g_m} (B_r \setminus B_{\frac{m}{2}}) + O(r^{\frac{3}{2}})$. Let $\tilde r$ be such that $\CL^3_{g_m} (\tilde \Omega_m) = \CL^3_{g_m} (B_{\tilde r} \setminus B_{\frac{m}{2}})$. Then $\tilde r = r + O (r^{- \frac{1}{2}})$. 

\item The Schwarzschild region $\tilde \Omega_m \subset M_m$ is $(\frac{1 + \tau}{2}, \frac{\eta}{2})$-off-center provided that $\CL^3_g(\Omega)$ is sufficiently large. Hence 
  \begin{equation*}
    A_m(\CL^3_{g_m}(\tilde \Omega_m)) +  \frac{\eta m \pi}{192}  \left(1 -  \frac{1}{\tau}  \right)^2 \tilde r \leq \H^2_{g_m} (\partial \tilde \Omega_m)
  \end{equation*}
by (\ref{eqn: comparison in Schwarzschild}).

\item $\H_{g_m}^2 (S_r) = A_m (\CL^3_{g_m} ( B_r \setminus B_{\frac{m}{2}})) \leq
A_m(\CL^3_{g_m}(\tilde \Omega_m)) +  O(\CL_g^3(\Omega)^{\frac{1}{6}})$ where the inequality follows by explicit computation (using (\ref{eqn:expansionprofile})) from 
\begin{equation*}
  \CL^3_{g_m} (B_r \setminus B_\frac{m}{2}) = \CL^3_g (\Omega) +
  O(\CL^3_g(\Omega)^{\frac{1}{2}}).
\end{equation*}
\item $\H^2_g(S_r) \leq \H^2_{g_m} (S_r) + O(1)$. This is obvious.

\item $\H^2_g (S_r) \leq \H^2_g(\partial \Omega) -  \frac{\eta m \pi}{200}  \left(1 -  \frac{1}{\tau}  \right)^2 r+ O(\CL^3_g(\Omega)^{\frac{1}{6}}) + O(\H^2_g(\partial \Omega)^{\frac{1}{4}})$.

\end{enumerate}
The conclusion follows from this since $\H^2_g(\partial \Omega) ^{\frac{1}{2}} \CL^3_g(\Omega)^{-\frac{1}{3}} \leq \Theta$.
\end{proof}

%%%%%%%%%%%%%%%%%%%%%%%%%%%%%%%%%%%%%%%%%%%%%%%%%%%%%%%%%%%%%%%%%%%%%%%%%%%%%%%%%%%%%%%
%%%%%%%%%%%%%%%%%%%%%%%%%%%%%%%%%%%%%%%%%%%%%%%%%%%%%%%%%%%%%%%%%%%%%%%%%%%%%%%%%%%%%%%
%%%%%%%%%%%%%%%%%%%%%%%%%%%%%%%%%%%%%%%%%%%%%%%%%%%%%%%%%%%%%%%%%%%%%%%%%%%%%%%%%%%%%%%

\section{Regularity of isoperimetric regions and the behavior of minimizing sequences} \label{sec:regularity}

In this section, we review the regularity theory for minimizers of area under a volume constraint in the presence of a smooth obstacle. The results discussed here are well-known and can be deduced from classical sources. For completeness and clarity, and because we have not been able to find a reference that includes our set up here completely, we supply a detailed outline of the argument along with further references, where more details on specific parts of the argument can be found. 

We consider an initial data set $(M, g)$ and its extension $(\hat M, \hat g)$ to a complete boundaryless Riemannian $3$--manifold, as in Section \ref{sec:definitions-notation}. Recall that the horizon $\partial M$, if non-empty, is the outermost minimal surface of $\hat M$.

\begin{proposition} An isoperimetric region containing the horizon has smooth, compact boundary. If this boundary intersects the horizon, then they coincide. 

\begin{proof} We first discuss the regularity of the reduced boundary $\partial^* \Omega$ away from the coincidence set $\supp (\partial^* \Omega) \cap \partial M$. 

A complete proof that $\partial^*\Omega$ has constant mean curvature away from the coincidence set is given in \cite[Proposition 2.1] {Duzaar-Steffen:1992}. This puts the monotonicity formula at one's disposal, and standard regularity analysis (see e.g. \cite[Theorem 2.5] {Duzaar-Steffen:1992}, which eventually refers to the classical paper \cite{Gonzalez-Massari-Tamanini:1983}) applies. The key points here are that there is no mass loss in the convergence (as sets of locally finite perimeter) of blow up sequences of $\partial^*\Omega$ at a point $x \in \supp(\partial^*\Omega) \setminus \partial M$, implying in conjunction with the monotonicity formula that the limiting objects are tangent cones, and that these tangent cones are area-minimizing boundaries and thus planes. In other words, the volume constraint scales away in the blow up limit. The proof of both these points proceeds as in the case of area-minimizing boundaries, cf. \cite{GMT}, applying for example the argument in \cite[Lemma 2.1]{Giusti:1981} to make effective use of the isoperimetric property of $\partial^*\Omega$. The regularity of $\partial^*\Omega$ near $x$ then follows at once from Allard's theorem. See also e.g. \cite{Almgren:1976, Gonzalez-Massari-Tamanini:1983, Morgan:2003} for alternative ways of arguing this step.

That $\partial^* \Omega$ is compact follows in a standard way from the monotonicity formula (see e.g. \cite[Lemma 10]{Brendle-Eichmair:2011}) and an explicit bound on $\H^2_g (\partial^* \Omega)$ that can be obtained from comparison; cf. \cite[Corollary 2.4]{Duzaar-Steffen:1992}. 

The regularity of $\partial^*\Omega$ along the horizon $\partial M$
follows from \cite{Giusti:1981}, \cite{Tamanini:1982}, and
\cite{Gonzalez-Massari-Tamanini:1983}, see also \cite[Theorem
1.3]{Huisken-Ilmanen:2001} and the references provided there. Again,
we outline the key points. We may assume that $\CL^3_g(\Omega \cap M)
>0$. Using that $\partial^*\Omega$ contains regular points, one
concludes that $\partial^*\Omega$ is almost minimizing in $\hat M$
(i.e. across the horizon) without volume constraint. It follows as
above that the mean curvature of $\partial^* \Omega$ is defined and
bounded along the coincidence set, that there is no mass loss in the
convergence of tangent blow up sequences at points in the coincidence
set, that the limits are cones, and finally that these cones are area-minimizing and thus planes. Hence $\partial^*\Omega$ is a $\C^{1,
  \alpha}$ surface near $\partial M$. 

The next step is to argue that the constant mean curvature of
$\partial^* \Omega$ away from the coincidence set, $H$, is
non-negative. If $H < 0$, then one could take the minimal area
enclosure of $\Omega$ in $M$ and use the same argument as above to
show that it is a smooth minimal surface away from the coincidence set
of $\partial^* \Omega$ with the horizon, where it is a priori only
$\mathcal{C}^{1, \alpha}$; cf. \cite[Theorem 1.3
(ii)]{Huisken-Ilmanen:2001}. The minimal area enclosure of $\Omega$ is
weakly mean-convex. The Harnack inequality shows that its components
either coincide with components of the horizon, or are disjoint from
the horizon. The latter scenario (for any component) contradicts our
assumption that the horizon is the outermost minimal surface in
$M$. We see that $H > 0$ unless $\partial \Omega = \partial M$. A
first variation argument shows that $\partial^*\Omega$ is weakly
mean-convex along the coincidence set. Again, we can conclude from the
Harnack inequality that the coincidence set is either empty or that
$\partial^*\Omega = \partial M$.
\end{proof}
\end{proposition}

We also want to understand the behavior of general minimizing sequences in initial data sets. The following proposition is a slight extension of a special case of \cite[Theorem 2.1]{Ritore-Rosales:2004}, see also \cite{Duzaar-Steffen:1992, Bray:1998} and the remarks below. 

\begin{proposition} \label{prop:cut_and_paste_for_minimizing_sequence} 
Given $V > 0$, there exists an isoperimetric region $\Omega \subset \hat M$ containing the
horizon and a radius $r \in [0, \infty)$ such that $\CL^3_g(\Omega) + \frac{4 \pi r^3}{3} = V$ and such that $\H^2_g(\partial \Omega) + 4 \pi r^2 = A_g(V)$. If $r>0$ and $\CL^3_g(\Omega) >0$, then the mean curvature of $\partial \Omega$ equals $\frac{2}{r}$. 

\begin{proof} By \cite[Theorem 2.1]{Ritore-Rosales:2004} and a simple rescaling argument, there exists an isoperimetric region $\Omega$ containing the horizon and a sequence of finite-perimeter Borel sets $\Omega_i$ diverging to infinity such that $\Omega \cap \Omega_i = \emptyset$, $\CL^3_g(\Omega) + \CL^3_g(\Omega_i) = V$, and $\H^2_g(\partial \Omega) + \lim_{i \to \infty} \H^2_g(\partial^* \Omega_i) = A_g(V)$. Applying the Euclidean isoperimetric inequality with a small fudge factor that tends to $1$ as $i \to \infty$ to the sets $\Omega_i$, we see that the sets $\Omega_i$ can be replaced by coordinate balls $B(p_i, r_i)$ of the same volume and such that $p_i \to \infty$. The observation about the mean curvature of the sphere that represents the runaway volume follows from a first variation argument. 
\end{proof}
\end{proposition}

Proposition \ref{prop:cut_and_paste_for_minimizing_sequence} leaves the possibility that part of the volume of a minimizing sequence for the isoperimetric problem slides to infinity. If this happens, the leftover isoperimetric limit is not a solution of the original problem. In Euclidean space, the situation is well-understood: for example, in \cite{Duzaar-Steffen:1992}, it is shown that to every closed curve in $\R^3$ and volume $V$ there exists a  mass-minimizing integer multiplicity current that bounds the curve while enclosing  oriented volume $V$ relative to a fixed filling of the curve. A key ingredient in the proof is the exact isoperimetric inequality for $\R^3$. It is used to argue that runaway volume can be clipped off and kept at fixed finite distance as a ball of the same volume, not increasing the area. A delicate cut and paste argument is developed in \cite[Sections 2.7 and 2.9]{Bray:1998} to show existence of isoperimetric regions on compact perturbations of Schwarzschild. In the proof, H. Bray uses an additional assumption (``Condition 1") in a subtle way to ensure that his isoperimetric Hawking mass is a monotone function of the volume. 

For later use, we state the following simple lemma. It follows readily from explicit comparison either with small geodesic balls, or with large coordinate balls:

\begin{lemma} \label{lem:quadraticareagrowthisoperimetric}
Let $(M, g)$ be an initial data set. There exists a constant $\Theta >
0$ so that for every isoperimetric region $\Omega$ containing the horizon one has that $\H^2_g(B_r \cap \partial \Omega) \leq \Theta r^2$ for all $r \geq 1$, and that $\H^2_g (\partial \Omega)^{\frac{1}{2}} \CL^3_g(\Omega)^{- \frac{1}{3}} \leq \Theta$ provided $\CL^3_g(\Omega) \geq 1$. 
\end{lemma}

%%%%%%%%%%%%%%%%%%%%%%%%%%%%%%%%%%%%%%%%%%%%%%%%%%%%%%%%%%%%%%%%%%%%%%%%%%%%%%%%%%%%%%%
%%%%%%%%%%%%%%%%%%%%%%%%%%%%%%%%%%%%%%%%%%%%%%%%%%%%%%%%%%%%%%%%%%%%%%%%%%%%%%%%%%%%%%%
%%%%%%%%%%%%%%%%%%%%%%%%%%%%%%%%%%%%%%%%%%%%%%%%%%%%%%%%%%%%%%%%%%%%%%%%%%%%%%%%%%%%%%%

\section{Large isoperimetric regions center} \label{sec:center}

\begin{theorem} \label{thm:centering} Let $(M, g)$ be an initial data set that is $\C^0$-asymptotic to Schwarzschild of mass $m > 0$. There exists a large constant $V_0 > 0$ with the following property: Let $\Omega$ be an isoperimetric region containing the horizon such that $\CL^3_g(\Omega) = V \geq V_0$. Let $r \geq 1$ be such that $\CL^3_g(B_r) = V$. Then $\partial \Omega$ is a smooth connected hypersurface close to the centered coordinate sphere $S_r$. The scale invariant $\C^2$-norms of functions that describe such large isoperimetric surfaces as normal graphs above the corresponding centered coordinate spheres tend to zero as the enclosed volume diverges to infinity.  
\end{theorem}
 
\begin{proof}
Let $\{\Omega_i\}_{i=1}^\infty$ be a sequence of isoperimetric regions containing the horizon and with $\CL^3_g(\Omega_i) \to  \infty$. In view of Lemma \ref{lem:quadraticareagrowthisoperimetric}, fixing parameters $(\tau, \eta) \in (1, \infty) \times (0, 1)$, we can apply Theorem \ref{thm:effective_volume_comparison} to $\Omega_i$ provided $i$ is sufficiently large. 

We consider the parts of the regions $\Omega_i$ that lie in $M \setminus B_1 \cong_x \R^3 \setminus B(0, 1)$. We use homotheties
$h_{\lambda} : x \to \lambda \cdot x$ in the Euclidean chart to scale down
by a factor $\lambda_i = \left( 3 \CL^3_g (\Omega_i \setminus B_1) /(4
  \pi)\right)^{\frac{1}{3}}$ to obtain sets $\hat \Omega_i \subset
\R^3 \setminus B(0, \lambda_i^{-1})$ that are locally isoperimetric
with respect to the metric $g_i:= \lambda_i^{-2} h_{\lambda_i}^*g$ and
such that $\CL^3_{g_i}(\hat \Omega_i) = \frac{4 \pi}{3}$. Note that $(\R^3 \setminus
B(0,\lambda_i^{-1}), g_i) \to (\R^3 \setminus \{0\}, \sum_{j=1}^3
dx_j^2)$ in $\C^2_{loc}$ and that $\CL_{g_i}^3 (B(0, 1)
\setminus B(0, \lambda_i^{-1})) \to \frac{4 \pi}{3}$. Passing to a
subsequence if necessary, we can assume that $\hat \Omega_i$ converges
locally as a set of finite perimeter to $\Omega$ in $\R^3$.

We claim that $\limsup_{i \to \infty}
\CL^3_{g_i}( \hat \Omega_i \setminus B(0, \tau)) = 0$ for every $\tau >1$. Suppose that
not. Passing to a subsequence if necessary, it follows that for some $\tau > 1$ and $\eps >0$ we have that $\CL^3_{g_i}( \hat \Omega_i \setminus B(0, \tau))\geq \eps$ for all $i$. The relative isoperimetric inequality, an appropriate version of which follows from (\ref{eqn:relativeSobolev}) in a standard way, gives that $\H^2_{g_i} (\partial \hat\Omega_i \setminus B(0, \tau)) \gtrsim \CL^3_{g_i}( \hat \Omega_i \setminus B(0, \tau))^{\frac{2}{3}}$. In particular, $\H^2_{g_i} (\partial \hat\Omega_i \setminus B(0, \tau))  \geq 2 \eta \H^2_{g_i} (\partial B (0, 1))$ for some $\eta >0$ and for all $i$. This implies that each $\Omega_i$ is $(\frac{1 + \tau}{2}, \eta)$-off-center. (The reason we are passing from $\tau$ to $\frac{1 + \tau}{2}$ and from $2 \eta$ to $\eta$ is that we have to adjust for the volume by $\CL^3_{g} (\Omega_i \setminus B_1) + \CL^3_g (B_1)$.) See Figure~\ref{fig:limit_off_center}. 
Theorem \ref{thm:effective_volume_comparison} shows that $\Omega_i$ is not isoperimetric, contradicting our assumption. 
Thus $\limsup_{i \to \infty}
\CL^3_{g_i}( \hat \Omega_i \setminus B(0, \tau)) = 0$ for every $\tau > 1$, as desired. It follows that $\Omega = B(0, 1)$. 

For isoperimetric regions, convergence as sets of locally finite perimeter is equivalent to locally
smooth convergence so long as the volume
does not shrink away. (See e.g. \cite[Proposition 5]{Ros:2005}.) It follows that, for $i$ large,  the boundary of $\Omega_i$
contains a component $\Sigma_i$ that is close to the centered coordinate sphere $S_{r_i}$ whose radius $r_i$ is such that $\CL_g^3(B_{r_i}) = \CL^3_g(\Omega_i)$. 

Let $\tilde \Omega_i$ be the bounded component of $M \setminus \Sigma_i$. We claim that $\Omega_i = \tilde \Omega_i$. 

To see that $\Omega_i \subset \tilde \Omega_i$, note first that the components of $\partial \Omega_i$ all have the same constant mean curvature $\sim 2 /r_i$. Assume that, after passing to a subsequence if necessary, every $\Omega_i$ has at least one component that is disjoint from $\tilde \Omega_i$. The preceding analysis shows that such components have to slide off to infinity in the preceding blow down limit. The monotonicity formula shows that the area of these components subconverges to a positive number in the blow down limit. It follows that $\limsup_{i \to \infty }\H^2_g (\Omega_i) \CL^3_g(\Omega_i)^{- 2 / 3}> (4 \pi) (4 \pi / 3)^{-2/3}$. On the other hand, a comparison with large coordinate balls gives that $\limsup_{i \to \infty }\H^2_g (\Omega_i) \CL^3_g(\Omega_i)^{- 2 / 3} \leq (4 \pi) (4 \pi / 3)^{-2/3}$. This contradiction shows that $\Omega_i \subset \tilde \Omega_i$.

Assume that $\Omega_i$ is properly contained in $\tilde \Omega_i$. The blow down argument shows that $\Omega_i \cap B_{\mu_i} = \tilde \Omega_i \cap B_{\mu_i}$ where $\mu_i \geq 1$ are such that $\mu_i / r_i \to 0$ as $i \to \infty$. Consider the region obtained from $\tilde \Omega_i$ by pushing its outer boundary $\Sigma_i$, which is close to a large coordinate sphere, inward until the resulting region has volume $\CL_g^3 (\Omega_i)$. The boundary area of this new region is strictly less than that of $\Omega_i$. This contradicts the assumption that $\Omega_i$ is an isoperimetric region. Thus $\Omega_i = \tilde \Omega_i$. 
\begin{figure}
  \centering
  \resizebox{0.4\linewidth}{!}{\input{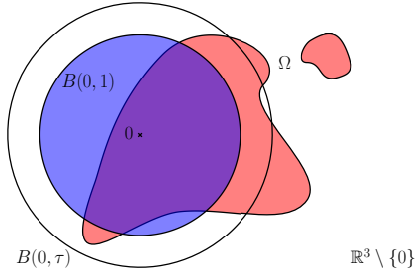}}
  \caption{If the blow down limit $\Omega$ has volume outside of $B(0, \tau)$, then the original sets $\Omega_i$ are $(\frac{1 + \tau}{2}, \eta)$-off-center for some $\eta >0$.}
  \label{fig:limit_off_center}
\end{figure}
\end{proof}

\begin{theorem} \label{thm:minimizers_exist} Let $(M, g) $ be an initial data set that is $\C^0$-asymptotic to Schwarzschild of mass $m>0$. There exists $V_0 >0$ so that for every volume $V \geq V_0$ there is an isoperimetric region $\Omega$ containing the horizon  with $\CL^3_g(\Omega) = V$. 
\begin{proof}
Let $V_i \to \infty$ be a divergent sequence of volumes. Let $r_i \geq 0$ be radii and $\Omega_i$ be isoperimetric regions containing the horizon as in Proposition \ref{prop:cut_and_paste_for_minimizing_sequence}. Using (\ref{eqn:expansionprofile}), we see that $\CL^3_g(\Omega_i) \to \infty$. From Theorem \ref{thm:centering} we know that for $i$ large, $\Omega_i$ is close to a large centered coordinate ball in $(M, g)$. If $r_i > 0$, then the mean curvature of $\partial \Omega_i$ and hence the radius of the coordinate sphere that it is close to correspond to that of $B(0, r_i) \subset \R^3$, by Proposition \ref{prop:cut_and_paste_for_minimizing_sequence}. A configuration of two large disjoint coordinate balls in $(M, g)$ of essentially the same radius is not isoperimetric, and far from it. Hence $r_i = 0$ for $i$ sufficiently large, and the theorem follows. 
\end{proof}
\end{theorem}

%%%%%%%%%%%%%%%%%%%%%%%%%%%%%%%%%%%%%%%%%%%%%%%%%%%%%%%%%%%%%%%%%%%%%%
%%%%%%%%%%%%%%%%%%%%%%%%%%%%%%%%%%%%%%%%%%%%%%%%%%%%%%%%%%%%%%%%%%%%%%
%%%%%%%%%%%%%%%%%%%%%%%%%%%%%%%%%%%%%%%%%%%%%%%%%%%%%%%%%%%%%%%%%%%%%%

\section {Isoperimetric regions in initial data sets with general asymptotics} \label{sec:general_asymptotics}

Let $(M, g)$ be an initial data set with non-negative scalar curvature. Let $\{\Omega_i \}_{i=1}^\infty$ be a sequence of isoperimetric regions containing the horizon such that $\CL^3_g(\Omega_i) \to \infty$. The argument in Proposition \cite[Proposition 5]{Ros:2005} shows that the $\Omega_i$ subconverge to a locally isoperimetric region $\Omega$. In Theorem \ref{thm:PMT} below, we show that the unbounded components of $\partial \Omega$ are totally geodesic and that the scalar curvature of $M$ vanishes on them. 

If we assume that the scalar curvature of $M$ is everywhere positive, this result puts a strong limitation on the possible behavior of large isoperimetric regions; cf. Corollary \ref{cor:PMT}.  

Theorem \ref{thm:PMT}  is the precursor of our more subtle result in \cite{stablePMT}, which applies to regions whose boundaries are only assumed to be volume-preserving stable constant mean curvature surfaces. The proofs of both results are based on ideas of R. Schoen and S.-T. Yau in \cite{Schoen-Yau:1979-pmt1}. On a technical level, the proof of Theorem \ref{thm:PMT} is quite different from that in \cite{stablePMT}, so we include it.

\begin {theorem} [Cf. Theorem 1.5 in \cite{stablePMT}]  \label{thm:PMT}
Assumptions as in the first paragraph. Then $\partial \Omega$ has at most one unbounded component, and this component is a totally geodesic area-minimizing hypersurface. The scalar curvature of $M$ vanishes on any unbounded component of $\partial \Omega$.  
\begin{proof} By \cite[Corollary 5.6]{stablePMT}, the mean curvature of $\partial \Omega_i$ tends to zero as $i \to \infty$. It follows that $\partial \Omega$ is a minimal surface. Let $\Sigma$ be an unbounded component of $\partial \Omega$. We employ an argument of R. Schoen and S.-T. Yau from \cite{Schoen-Yau:1979-pmt1} to show that $\Sigma$ is totally geodesic. We follow the main steps of \cite{Schoen-Yau:1979-pmt1} very closely and highlight our minor adaptations   to the present context. The important difference with \cite{Schoen-Yau:1979-pmt1} is that we don't a priori know that $\Sigma$ is strongly stable. That this is nevertheless the case follows from the result presented in Appendix \ref{sec:minimizing_limit}.
\begin{enumerate} [(a)]
\item By Appendix \ref{sec:minimizing_limit}, $\Sigma$ is an area-minimizing boundary. Hence $\Sigma$ is stable with respect to compactly supported variations: 
\begin{equation} \label{eqn: stability on sigma pmt}
\int _{\Sigma} \left( |h|^2 + \Ric_g(\nu, \nu) \right) \phi^2 d \H^2_g \leq \int_{\Sigma} |\nabla_{\Sigma} \phi|^2  d \H^2_g  \text{ for all } \phi \in \C_c^1(M).
\end{equation}

\item This step and the next one differ slightly from \cite{Schoen-Yau:1979-pmt1}.  The homothetic rescalings $\lambda^{-1} \left( \Sigma \setminus B_1 \right) \subset \R^3 \setminus B(0, \lambda^{-1})$ subconverge as $\lambda \to \infty$ to area-minimizing boundaries in $\R^3 \setminus \{0\}$ (with the Euclidean metric). Such boundaries are hyperplanes. It follows from (iterations of) this argument that $\Sigma$ intersects any sufficiently large coordinate sphere $S_r \subset M$ in a circle. It follows that $\Sigma$ is planar outside a compact subset of $(M, g)$. In particular, $\Sigma$ is of finite topological type.    

\item $\Sigma$ is a mass-minimizing integral current in $(M, g)$. The argument is indirect. Consider a large coordinate ball $B_r$ with mean-convex boundary $S_r$. From the preceding step we know that $S_r$ intersects $\Sigma$ transversely in a smooth connected curve. By the maximum principle argument of \cite{Solomon-White:1989}, the mass-minimizing current in $(M, g)$ spanning $\Sigma \cap S_r$ lies inside of $B_r$ and is disjoint from the horizon. By \cite{Hardt-Simon:1979}, this mass-minimizing integral current is a smooth, embedded, multiplicity one hypersurface. Its area must be $\H^2_g (\Sigma \cap B_r)$. It follows that $\Sigma$ is mass-minimizing with respect to current deformations, and not just amongst boundaries. We are grateful to Leo Rosales and to Brian White for helping us with this point.

\item $\Sigma$ has quadratic area growth: there exists a constant $\Theta>0$ depending only on $(M, g)$ such that $\H^2_g(\Sigma \cap B_r) \leq \Theta r^2$ for all $r \geq 1$ sufficiently large. This follows from the mass-minimizing property of $\Sigma$ and comparison with large coordinate spheres. 

\item We have that $\int_{\Sigma} |\Ric_g| d \H^2_g < \infty$. This follows from Lemma \ref{lem:area_growth_decay} and because $|\Ric_g| = O(r^{-3})$. 

\item Because $\Sigma$ has quadratic area growth, we can use the `logarithmic cut-off trick' in (\ref{eqn: stability on sigma pmt}) to obtain that $\int_{\Sigma} \left( |h|^2 + \Ric_g(\nu, \nu) \right) d \H^2_g < \infty$. It follows from the Gauss equation that $\int_\Sigma |\kappa| d \H^2_g < \infty$, where $\kappa$ is the Gauss curvature of $\Sigma$. 

\item From the Gauss equation and the Cohn-Vossen inequality \cite {Cohn-Vossen:1935}, one sees that  \begin{equation} \label{eqn:stabilityPMT} 0 \leq \int_{\Sigma} \left(\Scal_g + |h|^2\right) d \H^2_g \leq 2 \int_{\Sigma} \kappa d \H^2_g \leq 4 \pi \chi(\Sigma).\end{equation} (The Cohn-Vossen inequality applies because $\Sigma$ is complete, has absolutely integrable Gauss curvature, and is of finite topological type. See also \cite[p. 86]{Osserman:1986} and  \cite[p. 1]{Finn:1965}). The theorem will follow if we can show that $\int_\Sigma \kappa d \H^2_g = 0$. We will assume for a contradiction that  $\int_\Sigma \kappa d \H^2_g >0$. Note that in this case, (\ref{eqn:stabilityPMT}) implies that  $\Sigma$ is homeomorphic to the plane $\mathbb{C}$. 

\item Since $\int_\Sigma |\kappa| d \H^2_g < \infty$ and $\Sigma \cong \mathbb{C}$, a theorem of A. Huber's \cite{Huber:1957} gives that there exists a conformal diffeomorphism $F : \mathbb{C} \to \Sigma$. (We refer the reader to  \cite{Li-Tam:1991} for a comprehensive discussion of the topological type and the conformal structure of  complete surfaces the negative part of whose Gaussian curvature is integrable.) By results of R. Finn's \cite{Finn:1965} and A. Huber's \cite{Huber:1966}, one has that $4 \int_{\Sigma} \kappa d \H^2 = 4 \pi - \lim_{k \to \infty} A_k^{-1}L_k^2$ (the existence of the limit is part of the conclusion) where $L_k = \H^1_g (F(\{z \in \mathbb{C} : |z| =k \}))$ and $A_k := \H^2_g (F(\{z \in \mathbb{C} : |z| \leq k\}))$. The goal is to show that $\int_{\Sigma} \kappa d \H^2_g  = 0$. Since we already know that $0 \leq \int_{\Sigma} \kappa d \H^2_g$, this boils down to showing that $4\pi \leq \lim_{k \to \infty} A_k^{-1} L_k^2$.

\item Since $F$ is proper, $F(\{z \in \mathbb{C} : |z| =k \})$ will lie outside every given compact subset of $\Sigma$ (and hence of $M$) provided $i$ is sufficiently large. Hence we can view $F(\{z \in \mathbb{C} : |z| =k \})$ as a curve in Euclidean space $(\R^3, \delta_{ij})$ by using the coordinate system in the asymptotically flat end of $(M, g)$. Let $\tilde \Sigma_k$ be the least (Euclidean) area integer multiplicity current spanning $F(\{z \in \mathbb{C} : |z| =k \}) \subset \R^3$. There are two cases: either $\tilde \Sigma_k$ leaves every compact subset of $\R^3$ as $k \to \infty$, in which case $M_g(\tilde \Sigma_k) = (1 + o(1)) M_\delta (\tilde \Sigma_k)$ (where $M_{g}$ and $M_\delta$ denote the current mass with respect to $g$ and $\delta$ respectively), or there exists a radius $r_1 \geq 1$ such that $\tilde \Sigma_{k'} \cap B_{r_1} \neq \emptyset$  for a subsequence $\tilde \Sigma_{k'}$. Either way, the argument in \cite[p. 56-57]{Schoen-Yau:1979-pmt1} can be followed verbatim  to conclude that $4 \pi \leq \lim_ {k \to \infty} A_k^{-1} L_k^2$ and hence that $\int_{\Sigma} \kappa d \H^2_g= 0$. This contradicts our assumption and finishes the proof that $\Sigma$ is totally geodesic (and that $\partial \Omega$ cannot have unbounded components if $\Scal_g > 0$). 
\end{enumerate} 
It follows from the isoperimetric property and Lemma \ref{lem:parallelplanes} below that $\partial \Omega$ can only have one unbounded totally geodesic component.
\end{proof}
\end{theorem}

\begin{corollary} \label{cor:PMT} Let $(M, g)$ be an initial data set whose scalar curvature is everywhere positive. If the horizon $\partial M$ is empty, we also assume that there are no closed minimal surfaces in $M$. Let $\Omega_i \subset M$ be a sequence of isoperimetric regions enclosing the horizon whose volumes tend to infinity. Then $\limsup_{i \to \infty} \Omega_i := \bigcap_{j=1}^\infty \bigcup_{i=j}^\infty \Omega_i$ equals either $\partial M$ or $M$. 
\end{corollary}

\begin{lemma} [Essentially \protect{\cite[Proposition 3.1]{beig-schoen}}] \label{lem:parallelplanes}
  Let $(M,g)$ be an initial data set that satisfies the decay assumptions (\ref{eqn: decay assumptions integral decay}). There exist a radius $r_0 \geq 1$ and a constant $C \geq 1$ with the following property: If $\Sigma$ is a complete unbounded properly embedded totally geodesic surface in $M$, then
  $\Sigma \setminus B_{r_0}$ consists of finitely many components $\Sigma_1, \ldots, \Sigma_m$. Moreover, there exists a coordinate plane $P = \{(x_1, x_2, x_3) \in \R^3 \setminus B(0, 1): a x_1 + b x_2 + cx_3 = 0\}$ and functions $f_k: P \setminus B_{r_0} \to \R$ such that the graph of $f_k$ above $P \setminus B_{r_0}$ is contained in $\Sigma_k$ and such that
  \begin{equation*}
    ( \log r )^{-1} |f_k|  + r|\partial f_k| + r^2 |\partial^2 f_k| \leq C \text{ for every } k = 1, \ldots, m.
  \end{equation*}  
\end{lemma}

%%%%%%%%%%%%%%%%%%%%%%%%%%%%%%%%%%%%%%%%%%%%%%%%%%%%%%%%%%%%%%%%%%%%%%%%%%%%%%%%%%%%%%%
%%%%%%%%%%%%%%%%%%%%%%%%%%%%%%%%%%%%%%%%%%%%%%%%%%%%%%%%%%%%%%%%%%%%%%%%%%%%%%%%%%%%%%%
%%%%%%%%%%%%%%%%%%%%%%%%%%%%%%%%%%%%%%%%%%%%%%%%%%%%%%%%%%%%%%%%%%%%%%%%%%%%%%%%%%%%%%%

\appendix

\section{Integral decay estimates} \label{sec:integral_decay_estimates}

Our computations in this appendix take place on the part of an initial data set $(M, g)$ that is diffeomorphic to $\R^3 \setminus B(0, 1)$ and such that 
\begin{eqnarray} \label{eqn: decay assumptions integral decay} 
r|g_{ij } - \delta_{ij}| \leq C \text{ for all } r:= |x| \geq 1. 
\end{eqnarray} For Corollary \ref{lem:volume_comparison} we require in addition that 
\begin{eqnarray} 
\label{eqn: decay assumptions integral decay Schwarzschild}
r^2 |g_{ij} - \left(1 + \frac{m}{2r} \right)^4 \delta_{ij}| \leq C \text{ for all } r:= |x| \geq 1,
\end{eqnarray}
i.e. that $(M, g)$ is $\C^0$-asymptotic to Schwarzschild of mass $m>0$. 

\begin{lemma} \label{lem:area_growth_decay} Let $(M, g)$ be an initial data set. Let $r_0 \geq 1$. For every closed hypersurface $\Sigma \subset M$ such that $\H^2_g(\Sigma \cap B_r \setminus B_{r_0}) \leq \Theta r^2$ holds for all $r \geq {r_0}$ one has that
\begin{eqnarray*} 
\int_{\Sigma \setminus B_{r_0}} r^{-\gamma} d \H^2_g \leq  \frac{\gamma}{\gamma - 2} \Theta {r_0}^{2 - \gamma}
\end{eqnarray*} for every $\gamma>2$. 
\begin{proof} The proof uses the co-area formula exactly as in \cite[p. 52]{Schoen-Yau:1979-pmt1}. 
\end{proof}
\end{lemma}

\begin{corollary}  \label{cor:surface_comparison} Let $(M, g)$ be an initial data set. Let $r_0 \geq 1$. For every closed hypersurface $\Sigma \subset M$ such that $\H^2_g(\Sigma \cap B_r \setminus B_{r_0}) \leq \Theta r^2$ holds for all $r \geq {r_0}$ one has that 
\begin{eqnarray*}
\int_{\Sigma \setminus B_{r_0}} r^{-2} d \H^2_g \leq  r_0^{- \beta} \H^2_g(\Sigma \setminus B_{r_0})^{\frac{\beta}{2}} \left(\frac{2 \Theta}{\beta}\right)^{\frac{2 - \beta}{2}}
\end{eqnarray*} for every $\beta \in (0, 2)$. 
\end{corollary}

\begin{lemma} \label{lem:volume_comparison} Let $(M,g)$ be an initial data set satisfying (\ref{eqn: decay assumptions integral decay Schwarzschild}). There is a constant $C' \geq 1$ depending only on $C$ such that for every ${r_0} \geq 1$ and every bounded measurable subset $\Omega \subset M$ one has that
\begin{eqnarray*}  |\CL^3_g(\Omega \setminus B_{r_0}) - \CL^3_{g_m}(\Omega \setminus B_{r_0})| \leq C' \left( \frac{3-\alpha}{\alpha - 1}\right)^{\frac{3 - \alpha}{3}} \CL^3_{g}(\Omega \setminus B_{r_0})^\frac{\alpha}{3} r_0^{1- \alpha} \end{eqnarray*} for every $\alpha \in (1, 3)$. 

\begin{proof}
The volume elements differ by terms $O(r^{-2})$. The estimate follows from the H\"older inequality and the fact that 
\begin{equation*}
  \int_{\R^3 \setminus B(0, {r_0})} r^{\frac{-2 \alpha}{3 - \alpha}} d \CL^3_\delta = \frac{3-\alpha}{3(\alpha - 1)} r_0^{\frac{3(1-\alpha)}{3 - \alpha}}
\end{equation*}
for $\alpha \in (1, 3)$. 
\end{proof}
\end{lemma}

%%%%%%%%%%%%%%%%%%%%%%%%%%%%%%%%%%%%%%%%%%%%%%%%%%%%%%%%%%%%%%%%%%%%%%%%%%%%%%%%%%%%%%%
%%%%%%%%%%%%%%%%%%%%%%%%%%%%%%%%%%%%%%%%%%%%%%%%%%%%%%%%%%%%%%%%%%%%%%%%%%%%%%%%%%%%%%%
%%%%%%%%%%%%%%%%%%%%%%%%%%%%%%%%%%%%%%%%%%%%%%%%%%%%%%%%%%%%%%%%%%%%%%%%%%%%%%%%%%%%%%%

\section {Further results in H. Bray's thesis} \label{sec:Bray_thesis}

\subsection {$u_c(s)$ is increasing}
For the convenience of the reader, we reproduce H. Bray's proof that the function $u_c (s)$ in Section \ref{sec:volume_comparison} is increasing using our notation. This fact is all that H. Bray needed to show that the isoperimetric surfaces in Schwarzschild are the centered spheres.
\begin{lemma} [\protect{\cite[Lemma 2]{Bray:1998}}]
  Let $c>0$ and let $g_m^c = u_c(s)^{-2}ds^2 + u_c(s) s^2 g_{\S^2}$ be a smooth metric on $[c, \infty) \times \mathbb{S}^2$ with $u_c(c)=\alpha < 1$ and $\partial_s u_c|_c = 0$. Assume that $([c, \infty) \times \S^2, g_m^c)$ is isometric to the mean concave exterior region that lies beyond an umbilic constant mean curvature sphere of area $\alpha 4 \pi c^2$ in Schwarzschild of mass $m >0$. Then $u_c(s)\in(\alpha,1)$ for $s>c$ and $u_c$ is increasing in $s$.
\end{lemma}
\begin{proof}
  The Hawking mass is  a first integral for the second order ordinary differential equation that $u_c$ is required to satisfy so that the metric 
  $g_m^c = u_c^{-2} ds^2 + u_c s^2 g_{\S^2}$ is scalar flat. Up to a
  positive multiplicative constant, the Hawking mass of $\{s\} \times \S^2$ with respect to $g_m^c$ is given by
  \begin{equation}
    y(s) \left( 1 - \frac{y(s)^4 y'(s)^2}{s^4}\right).  \label{eqn: Hawking mass}
  \end{equation}
  Here, $y(s) := \sqrt {u_c(s) s^2}$ and  the prime denotes differentiation with respect to $s$. It follows that $(s^3 - y^3)' \geq
  0$ and hence $1 - \left(\frac{c}{s}\right)^3 (1 - \alpha^\frac{3}{2})
  \geq u_c(s)^\frac{3}{2}$. We see that $u_c(s) < 1$ for all $s \geq c$. Assume that there is an $s \in [c, \infty)$ such that $u_c'(s) = 0$. For such $s$, we have that
  $y'(s) = \sqrt{u_c (s)}$ and $y''(s) = \frac{s}{2 \sqrt{u_c(s)}} u_c''(s)$. Differentiating the constant Hawking mass (\ref{eqn: Hawking mass}), we obtain that
  \begin{equation*}
    y''(s) = \frac{(1 - u_c(s)^3) s^{\frac{3}{2}}}{2 u_c(s)^{\frac{5}{2}}}.
  \end{equation*}
  Since we already know that $u_c(s) < 1$, we obtain that $y''(s) > 0$ and hence $u_c''(s) >0$ for every $s \geq c$ such that $u_c'(s) = 0$. This implies that $u_c(s)$ is increasing.  
\end{proof}

\subsection{Isoperimetric surfaces in compact perturbations of \\ Schwarzschild}
\label{sec:isop-surf-comp}
In \cite[Section 2.6]{Bray:1998}, H. Bray shows that in an initial data set
$(M,g)$ that is identically Schwarzschild outside a compact set, the large umbilic constant mean curvature spheres are isoperimetric surfaces for the volume they enclose with the horizon of $(M, g)$.  We briefly outline H. Bray's argument: 

Fix a large centered coordinate sphere $S_r$ that lies in the Schwarzschild part
of the manifold. As explained in Section \ref{sec:volume_comparison}, one can construct a manifold $((0, \infty) \times \mathbb{S}^2 , g_m^c)$ by gluing the tip of a cone to the sphere $S_r$ in Schwarzschild in such a way that both the metric and the mean curvature match. H. Bray then constructs an area non-increasing map $\phi: (M,g) \to ((0, \infty) \times \mathbb{S}^2,g_m^c)$ so that the sphere $S_r$ in $(M,g)$ is mapped
isometrically onto $\{c\} \times \mathbb{S}^2$ in $((0, \infty) \times \mathbb{S}^2 ,g_m^c)$ and such that
$\phi$ is an isometry outside of
$S_r$ (respectively $\{c\} \times \mathbb{S}^2$). The construction of $\phi$ on the remainder of $M$ starts at $S_r$ and proceeds
inwards incrementally. In the spherically symmetric part of $(M,g)$, one chooses
$\phi$ to be also spherically symmetric and such that $\phi$ decreases area as little as possible. This requirement leads to an ordinary differential equation for
the stretching of the spheres that lie inside of $S_r$. The analysis of this ordinary differential equation
then yields that if $r$ is sufficiently large, a certain sphere that is still outside the compact perturbation and hence in the spherically symmetric part of $(M, g)$ gets
mapped to the tip of the cone. In particular, all of the non-Schwarzschild part of $(M, g)$ is mapped to the vertex of the
cone.

Since $\phi$ is area non-increasing inside of $S_r$, it is also volume
non-in\-creas\-ing. This implies that any other surface $\Sigma$ in
$(M,g)$ which contains at least as much volume as $S_r$ has larger
area: use $\phi$ to map $S_r$ and $\Sigma$ to $((0, \infty) \times \mathbb{S}^2,g_m^c)$ and use
 that $S_r$ is (outer) isoperimetric in $((0, \infty) \times \mathbb{S}^2,g_m^c)$.

H. Bray's technique to identify the isoperimetric surfaces of Schwarzschild has been generalized to a certain class of rotationally symmetric manifolds in \cite{Bray-Morgan:2002}, and further in \cite{Maurmann-Morgan:2009}. See also the comment before the statement of Theorem 2.6 in \cite{Maurmann-Morgan:2009} for a clarification of the hypotheses in \cite{Bray-Morgan:2002}.

%%%%%%%%%%%%%%%%%%%%%%%%%%%%%%%%%%%%%%%%%%%%%%%%%%%%%%%%%%%%%%%%%%%%%%%%%%%%%%%%%%%%%%%
%%%%%%%%%%%%%%%%%%%%%%%%%%%%%%%%%%%%%%%%%%%%%%%%%%%%%%%%%%%%%%%%%%%%%%%%%%%%%%%%%%%%%%%
%%%%%%%%%%%%%%%%%%%%%%%%%%%%%%%%%%%%%%%%%%%%%%%%%%%%%%%%%%%%%%%%%%%%%%%%%%%%%%%%%%%%%%%
  
\section{Remark on locally isoperimetric surfaces}  \label{sec:minimizing_limit}

In this appendix we show that an unbounded minimal surface is area-minimizing if it is the smooth limit of isoperimetric surfaces. This observation is used in Section \ref{sec:general_asymptotics}. 

The proof follows from the same (classical) techniques that establish the regularity of isoperimetric surfaces. We include details for completeness and clarity. We deliberately phrase the proof in non-technical terms  to help those readers who are not experts in geometric measure theory. \\

The regularity of rectifiable boundaries that minimize area with respect to a volume constraint was established in \cite{Giusti:1981} and \cite{Gonzalez-Massari-Tamanini:1983}. Implicitly, this result is already contained in  \cite{Almgren:1976}; cf. the remarks in the introduction of \cite{Morgan:2003}. The papers \cite{Giusti:1981, Gonzalez-Massari-Tamanini:1983} both rely on De Giorgi's method. The ways they go about dealing with the volume constraint are very different, however. In \cite{Giusti:1981}, a perturbation vector field is used to adjust the volume of a region by a small given amount while changing the area of its boundary in a controlled way. Morally, we follow the approach of \cite{Giusti:1981} closely in this appendix. In \cite{Gonzalez-Massari-Tamanini:1983}, it is shown that there exist open balls in an isoperimetric region as well as its complement, so that small volume can be added or deleted in a controlled way. 

We refer the reader to \cite{Morgan:2003}, in particular to Proposition 3.1 therein, for the development of the regularity theory for isoperimetric surfaces in Riemannian manifolds. There, further important references to the literature can be found. We also refer the reader to the paper \cite{Morgan-Ros:2010}, which contains useful observations regarding locally isoperimetric regions and additional references.    \\
 
Let $(M, g)$ be an initial data set as in Definition \ref{def:initial_data_sets}. In particular, $(M, g)$ is homogeneously regular, i.e. its curvature tensor is bounded and its injectivity radius is bounded below; cf. \cite[remarks below Theorem 3]{Schoen-Simon:1981}.  

Let $\Omega \subset M$ be a smooth region that minimizes area with respect to compactly supported volume-preserving deformations, i.e. for every region $\Omega' \subset M$ such that $\Omega  \Delta \Omega' \Subset B$ where $B$ is bounded and open in $M$ and such that $\CL^3_g (B \cap \Omega) = \CL_g^3 (B \cap \Omega')$ one has that $\H^2_g(B \cap \partial \Omega) \leq \H^2_g(B \cap \partial \Omega')$. In addition, we assume that $\partial \Omega$ is minimal and unbounded. 

Using the monotonicity formula in the form \cite[Section 5]{Schoen-Simon:1981} and elementary comparison arguments, one obtains that 
\begin{eqnarray}
\H^2_g(\partial \Omega) &=& \infty, \nonumber \\
\limsup_{s \to \infty} s^{-2} \H^2_g(B_s \cap \partial \Omega) &<& \infty, \nonumber  \\ 
\liminf_{s \to \infty} \frac{\H^2_g (B_{s+1} \cap \partial \Omega )}{ \H^2_g (B_s \cap \partial \Omega)} &=& 1. \label{volumegrowthestimate}
\end{eqnarray}
Given $s, r \geq 1$ with $s \geq r$, we let $A_{r, s}:= B_s \setminus B_r$.   

\begin{proposition} Let $(M, g)$ and $\Omega \subset M$ be as above. Then $\partial \Omega$ is area-minimizing. 
\begin{proof}
Let $\nu$ be the outward unit normal  field of $\partial \Omega$. Let $\exp$ denote the exponential map of $(M, g)$. A variation of the proof of \cite[Proposition 5]{Ros:2005} shows that the curvature of $\partial \Omega$ is bounded and that there exists $\varepsilon \in (0, \frac{1}{2})$ small such that the map $E : \partial \Omega \times (- \varepsilon, \varepsilon) \to M$ defined by $E(\sigma, t) = \exp_\sigma (t \nu(\sigma))$ is a diffeomorphism with its image. The constant $\varepsilon \in (0, \frac{1}{2})$ can be chosen so that for some $C\geq 1$, the following holds: 

Let $f \in \C_c^1 (\partial \Omega)$ be such that $0 \leq f \leq \varepsilon /2$. Let $ \Omega_f $ be the compact region bounded by $\{\exp_\sigma f(\sigma) : \sigma \in \partial \Omega\}$. Let $W_f := \{x \in M : \dist_g(x, \supp (f)) < 1\}$. Then 
\begin{eqnarray}
&& C  \H^2_g (W_f\cap \partial \Omega)\sup_{\sigma \in \partial \Omega }  f (\sigma)  \geq       \CL^3_g (W_f \cap  \Omega_f ) - \CL^3_g (W_f \cap  \Omega ) \nonumber   \\ 
& \geq&  \frac{1}{C} \H^2_g(U  \cap \partial \Omega) \inf_{\sigma \in U \cap \partial \Omega} f (\sigma) \text{ for every open } U \subset \supp (f).  \label{eqn:change_of_volume} 
\end{eqnarray}
Using also that $\partial \Omega$ is minimal,  
\begin{eqnarray} \label{eqn:change_of_area}
| \H^2_g ( W_f \cap \partial \Omega_f ) - \H^2_g (W_f\cap \partial \Omega) |   \leq C \int_{\partial \Omega} (f^2 + |\nabla f|^2).
\end{eqnarray}

Let $\Omega' \subset M$ be a region with $\Omega' \Delta \Omega \Subset M$. Let $r \geq 1$ be such that  $\Omega' \Delta \Omega \subset B_r $. Let $$ \Delta V := \CL^3_g( \Omega \cap B_r) - \CL^3_g(\Omega' \cap B_r).$$ Assume that $\Delta V>0$. (The discussion when $\Delta V < 0$ is analogous.) 

Since $\H^2_{g} (\partial \Omega) = \infty$, we have that $\frac{\varepsilon}{3 C}\H^2_g ( A_{r+2, s-1} \cap \partial \Omega) \geq \Delta V$ if $s$ is sufficiently large. By (\ref{volumegrowthestimate}), there exists a sequence $s_i \to \infty$ such that the quotients of $\H^2_g(A_{r, s_i+1} \cap \partial \Omega)$ and $\H^2_g(A_{r+2, s_i-1} \cap \partial \Omega)$ are close to one. 

Given $\delta_i \in (0, \varepsilon/2)$, let $f \in \C_c^2(\partial \Omega)$ with $\supp (f) \subset A_{r+1, s_i}$ be such that $0 \leq f \leq \delta_i$ and $|\nabla f| \leq 2 \delta_i$, and such that  $f = \delta_i$ on $A_{r+2, s_i-1}$. Using (\ref{eqn:change_of_volume}) with $U = A_{r+2, s_i-1}$, we see that $\CL^3_g ( \Omega_f \cap A_{r, s_i+1} ) - \CL^3_g ( \Omega \cap A_{r, s_i+1}) =  \Delta V$ for a choice of $\delta_i  \sim \frac{\Delta V}{\H^2_g( A_{r+1, s_i} \cap \partial \Omega)}$. From (\ref{eqn:change_of_area}), we obtain that $$ |\H^2_g ( B_{s_i +1} \cap \partial \Omega_f ) - \H^2_g (B_{s_i+1} \cap \partial \Omega) | = O \left(\frac{(\Delta V)^2}{\H^2_g(A_{r+1, s_i} \cap \partial \Omega )}\right).$$ 
Let $\tilde \Omega \subset M$ be the region such that $\tilde \Omega  \setminus B_r = \Omega_f \setminus B_r$ and such that $\tilde \Omega \cap B_r = \Omega' \cap B_r$. Then $\CL^3_g( B_{s_i+1} \cap \tilde \Omega) = \CL_g^3 (B_{s_i+1} \cap \Omega)$. Then
\begin{equation*}
  \begin{split}
    \H^2_g (B_{s_i+1} \cap \partial \Omega)
    &\leq
    \H^2_g (B_{s_i+1} \cap \partial \tilde \Omega)
    \\
    &= \H^2_{g} (B_{s_i+1} \cap \partial \Omega ') + O
    \left(\frac{(\Delta V)^2}{\H^2_g(A_{r+1, s_i} \cap \partial \Omega)
      }\right). 
  \end{split}
\end{equation*}
The last term  tends to zero as $i \to \infty$. Since $\Omega \setminus B_r = \Omega' \setminus B_r$, we see that $\H^2_g(B_r \cap \partial \Omega) \leq \H^2_g (B_r \cap \partial \Omega')$. 
\end{proof}
\end{proposition}

\bibliographystyle{amsplain}
\bibliography{references}

\end{document}